\documentclass[10pt,reqno]{amsart}

\usepackage[all]{xy}
\usepackage{mathpazo}

\usepackage{ifpdf}
\ifpdf
	\usepackage[pdftex]{graphicx}
\usepackage[pdftex,plainpages=false,pdfpagelabels,pdfstartview=FitH]{hyperref}
\else
	\usepackage[dvips]{graphicx}
	\usepackage[hypertex]{hyperref}
\fi

\graphicspath{{Pics/}}

\makeatletter
\@addtoreset{equation}{section}
\renewcommand{\theequation}{\thesection.\@arabic\c@equation}
\newcommand\Nopagebreak{\@nobreaktrue\nopagebreak}
\makeatother

\newtheorem{thm}{Theorem}[section]

\newtheorem{lemm}[thm]{Lemma}

\theoremstyle{definition}
\newtheorem{defn}[thm]{Definition}
\newtheorem{rem}[thm]{Remark}
\newtheorem{rems}[thm]{Remarks}

\allowdisplaybreaks
\def\ve{\varepsilon}
\def\ip#1{\left\langle #1\right\rangle}
\def\nm#1{\left| #1\right|}

\def\D{\mathrm{d}}
\newcommand{\at}[2]{\left.#1\right|_{#2}}

\def\ad{\operatorname{ad}}
\def\rank{\operatorname{rank}}
\def\diag{\operatorname{diag}}
\def\sgn{\operatorname{sgn}}
\def\Ric{\operatorname{Ric}}
\def\twovec#1#2{\begin{pmatrix}#1\\#2\end{pmatrix}}
\def\stwovec#1#2{\left(\begin{smallmatrix}#1\\#2\end{smallmatrix}\right)}

\def\R{\mathbb{R}}
\def\C{\mathbb{C}}
\def\pr{\mathbb{P}}
\def\I{\mathrm{I}}
\def\II{\mathbb{I}}
\def\rO{\mathrm{O}}
\def\cL{\mathcal{L}}
\def\cN{\mathcal{N}}
\def\cM{\mathcal{M}}

\def\fo{\mathfrak{o}}
\def\fu{\mathfrak{u}}
\def\fk{\mathfrak{k}}
\def\fa{\mathfrak{a}}

\def\fp{\mathfrak{p}}
\def\fgl{\mathfrak{gl}}
\def\cl#1{\overline{#1}}
\def\Id{\mathrm{Id}}
\def\V#1{\mathbf{#1}}
\def\lst#1#2{{#1}_1,\ldots,{#1}_{#2}}
\def\vect#1#2{\lst{\V{#1}}{#2}}

\def\l{\lambda}
\def\n{\, \vert\, }
\def\e{\epsilon}
\def\a{\alpha}
\def\cO{\mathcal{O}}
\def\ti{\tilde}
\newcommand{\bpm}{\begin{pmatrix}}
\newcommand{\epm}{\end{pmatrix}}

\def\otwonone{\frac{\rO(2n-1,1)}{\rO(n)\times \rO(n-1,1)}}

\def\li{\langle}
\def\ri{\rangle}
\def\g{\gamma}

\def\bb{{\bf b}}
\def\bc{{\bf c}}
\def\d{\delta}
\def\rd{{\rm d\/}}

\newenvironment{smatrix}{\left(\begin{smallmatrix}}{\end{smallmatrix}\right)}

\def\mathllap{\mathpalette\mathllapinternal}
\def\mathrlap{\mathpalette\mathrlapinternal}

\def\mathllapinternal#1#2{%
             \llap{$\mathsurround=0pt#1{#2}$}}
\def\mathrlapinternal#1#2{%
             \rlap{$\mathsurround=0pt#1{#2}$}}

\begin{document}

\title{Conformally flat submanifolds in spheres and  integrable systems}
\author{Neil Donaldson$^\dag$}\thanks{$^\dag$Research supported
in  part by NSF Advance Grant}
\address{Department of Mathematics, UCI, Irvine, CA 92697-3875}
\email{ndonalds@math.uci.edu}
\author{Chuu-Lian Terng$^\ast$}\thanks{$\ast$Research supported
in  part by NSF Grant DMS-0707132}
\address{Department of Mathematics, UCI, Irvine, CA 92697-3875}
\email{cterng@math.uci.edu}

\maketitle
\centerline{\today}

\begin{abstract}
\'E. Cartan proved that conformally flat hypersurfaces in $S^{n+1}$ for $n>~3$ have at most two distinct principal curvatures and locally envelop a one-parameter family of $(n-1)$-spheres. We prove that the Gauss-Codazzi equation for conformally flat hypersurfaces in $S^4$ is a soliton equation, and use a dressing action from soliton theory to construct geometric Ribaucour transforms of these hypersurfaces.  We describe the moduli of these hypersurfaces in $S^4$ and their loop group symmetries. We also generalise these results to conformally flat $n$-immersions in $(2n-2)$-spheres with flat and non-degenerate normal bundle.
\end{abstract}

\section{Introduction}
 
An immersion $f:M^n\to (N,g)$ is \emph{conformally flat} if there exists a flat metric in the conformal class of the induced metric $f^*g$: that is there exists a smooth function $u:M\to\R$ such that $e^{2u}f^\ast g$ is flat. This condition is equivalent to the Weyl tensor of $f^\ast g$ being zero when $n>3$, and to the Schouten tensor $S=\Ric(f^\ast g)-\frac R 4 f^\ast g$ being a Codazzi tensor when $n=3$ (i.e., $\nabla S$ is a symmetric $3$-tensor). 

The history of conformally flat immersions is long, the search for conformally flat submanifolds being a natural task in conformal geometry. The study of conformally flat \emph{hypersurfaces} in $S^{n+1}$ dates back to Cartan \cite{Cartan1917} who, demonstrated that the only such hypersurfaces for $n>3$ are the channel hypersurfaces: envelopes of a 1-parameter family of $(n-1)$-spheres. In particular these have (at most) two distinct principal curvatures. For $n=2$ the problem is uninteresting as every surface is conformally flat. For $n=3$, however, there are more varied conformally flat hypersurfaces: not only are there the channel examples, there also exist hypersurfaces with 3 distinct curvatures. These were first discussed by Hertrich--Jeromin \cite{Hertrich-Jeromin1994,Hertrich-Jeromin1996}, who moreover described the link between conformally flat hypersurfaces in $S^4$, curved flats in the space of circles in $S^4$, triply orthogonal systems, and Guichard nets. The classification of these hypersurfaces, however, remained unknown. 
 
Given a conformally flat immersion $f:M^n\to S^{2n-2}$ with flat normal bundle, we embed $S^{2n-2}$ naturally in the \emph{light-cone} $\cL^{2n-1,1}$ of isotropic vectors in a Lorentzian $\R^{2n-1, 1}$, and construct a \emph{flat lift} $F:M^n\to\cL^{2n-1,1}$: this $F$ is immersed, has flat induced metric, and flat normal bundle. Since both are flat, the tangent and normal bundle decomposition of the trivial $\R^{2n-1,1}$-bundle is a \emph{curved flat} \cite{Ferus1996} in the pseudo-Riemannian symmetric space $U/K=\otwonone$ (the Grassmannian of space-like $n$-planes in $\R^{2n-1, 1}$). We are thus immediately in the realm of integrable systems. 
  
Curved flats in $U/K$ with a \emph{good} co-ordinate system give rise to Terng's $U/K$-system (\cite{Terng1997}), which is constructed as follows: suppose that $\tau$ is the involution of the Lie group $U$ defined by $\tau(g)= \I_{n, n}g\I_{n,n}^{-1}$, where $\I_{n,n}=\begin{smatrix}\I_n&0\\0&-\I_n\end{smatrix}$ and $\I_n$ is the $n\times n$ identity matrix. Then $K$ is the fixed point set of $\tau$ in $U$. Let $\fu=\fk+\fp$ denote the $\pm$-eigenspace decomposition of $\D\tau_e$. $K$ then acts on $\fp$ by conjugation. Let $\fa$ be a maximal abelian subalgebra in $\fp$, and $\{\lst{a}{n}\}$ a basis of $\fa$. \emph{The $U/K$-system defined by $\fa$} is the following system for $\Xi:\R^n\to\fa^\perp\cap\fp$:
\begin{equation} \label{eq:uk}
[a_i,\Xi_{x_j}]-[a_j,\Xi_{x_i}]-[[a_i,\Xi],[a_j,\Xi]]=0,\quad i\neq j,
\end{equation}
where $\fa^\perp$ is the orthogonal complement of $\fa$ with respect to the Killing form. 

Unlike in the Riemannian symmetric case, not all maximal abelian subalgebras in $\fp$ are conjugate under $K$; there are both semi-simple and non-semisimple such subalgebras. We note that two maximal abelian subalgebras in $\fp$ conjugate under $K$ give rise to equivalent $U/K$-systems, while two non conjugate maximal abelian subalgebras in $\fp$ give non-equivalent $U/K$ systems. 

The normal bundle of an immersion is termed \emph{non-degenerate} if the dimension of the space of shape operators at each point is equal to the co-dimension. An immersion has \emph{uniform multiplicity one} if it has flat normal bundle and distinct curvature normals (equivalently all curvature distributions have rank one). It follows from the definition that an $n$-dimensional submanifold in $\R^{2n-1,1}$ with flat and non-degenerate normal bundle has uniform multiplicity one. We prove that a conformally flat $n$-immersion into $S^{2n-2}$ with uniform multiplicity one gives rise to a flat $n$-immersion in the light-cone $\cL^{2n-1,1}$ with flat non-degenerate normal bundle, and that the converse is also true. To study conformally flat $n$-immersions in $S^{2n-2}$ with uniform multiplicity one is thus the same as to study flat $n$-immersions in $\cL^{2n-1,1}$ with flat non-degenerate normal bundle. We show that there exist line of curvature co-ordinate systems for these flat immersions, and that their Gauss-Codazzi equations amount to the $U/K=\otwonone$-system defined by a semi-simple maximal abelian subalgebra $\fa$. Conversely, given a solution to the $U/K$-system defined by $\fa$, and a null vector $\bc\in\R^{n-1,1}$ we obtain a a conformally flat immersion in $S^{2n-2}$ with uniform multiplicity one.

Motivated by definitions in classical differential geometry, we call  a diffeomorphism $\phi:M\to \ti M$ between $n$-immersions in space forms with flat normal bundle a {\it Combescure transform\/} if $\phi$ maps principal directions of $M$ to those of $\ti M$ and they are parallel. A Combescure transform $\phi$ is {\it Christoffel\/} if it is orientation reversing.  Given a solution to the $U/K$-system defined by $\fa$ and a null vector $\bc\in \R^{n-1,1}$ with $\bc^t\bc= 2$, we construct a flat $n$-immersion $F_\bc$ in $\cL^{2n-1,1}$ with flat normal and non-degenerate bundle. Hence $F_\bc$ projects to a conformally flat immersion in $S^{2n-2}$ with uniform multiplicity one. Moreover, if $\bc$ and $\bb$ are null vectors with Euclidean length $\sqrt{2}$, then $F_\bc(x)\mapsto F_\bb(x)$ is a Combescure transform.

Because of the correspondence between solutions of the $U/K$-system and conformally flat $n$ immersions in $S^{2n-2}$ with flat and non-degenerate normal bundle, all the machinery of soliton theory applies: loop-group dressing of solutions to obtain new conformal flats or simply dressing vacuum solutions to obtain more complex explicit conformally flat immersions; existence results such as Cartan-K\"ahler and inverse scattering, etc. In particular:
\begin{enumerate}
\item We may dress solutions by special, simple elements, whose action may be calculated explicitly by residues. The action of such elements is seen to be by \emph{Ribaucour transforms} on conformal flats: corresponding immersions envelop (have first-order contact with) a congruence of $n$-spheres in such a way that principal curvature directions on the envelopes correspond under the congruence.
\item Local analytic conformally flat $n$-immersions in $S^{2n-2}$ are determined by $n^2-n$ functions of one variable.
\item The Cauchy problem for the $U/K$-system with rapidly decaying initial data on a regular line can be solved globally.\footnote{Recall that $a\in\fa$ is \emph{regular} if $\ad(a):\fa^\perp\cap\fp\to\fk$ is a injective.} Although the resulting $n$ dimensional submanifolds may have cusp singularities, the frame is globally defined and smooth. 
\item The moduli space of such immersions has a loop group symmetry.
\end{enumerate}

If the normal bundle is degenerate and the curvature distributions $E_i$ (common eigenspaces of the shape operators) have constant ranks, then we show that all but one of the $E_i$s have rank one. Such submanifolds are thus envelopes of $k$-parameter families of $(n-k)$-spheres. If, in addition, these immersions are assumed to have line of curvature co-ordinates, then the Gauss-Codazzi equations are the $U/K$-system defined by a non-semisimple maximal abelian subalgebra $\fa$ in $\fp$. Conversely, given a solution of the $U/K$-system defined by $\fa$, we obtain an $(n-2)$-parameter family of flat lifts, each of which gives rise to a conformally flat immersion in $S^{2n-2}$ with flat normal bundle, but not with uniform multiplicity one. When $n=3$, these give \emph{channel immersions}. Loop group dressing still works, and we can construct channel immersions from any germ of an $\fo(2n-1,1,\C)$-valued holomorphic map at $\l=\infty$ that satisfies the reality condition associated to $U/K$.

Most of the results for conformally flat $n$-immersions in $S^{2n-2}$ with uniform multiplicity one hold for conformally flat $n$-immersions in $S^{2n-2+k}$ with flat normal bundle, $n$ curvature normals such that the orthogonal complement of the subbundle spanned by $n$ curvature normals is flat. There exist line of curvature co-ordinates and a correspondence between such immersions and solutions of the $\frac{\rO(2n+k-1, 1)}{\rO(n)\times \rO(n+k-1,1)}$-system. 

The paper is organised as follows. In section \ref{sec:flatlift}, we generalise Hertrich--Jeromin's work on conformally flat immersions of hypersurfaces in $S^4$ to $n$-submanifolds in $S^{2n-2}$: we outline the light-cone model and how conformal flat immersions in the sphere correspond to genuine flat immersions in the light-cone, then consider the curvature distributions of corresponding maps, and how their fundamental forms compare. The link between conformal flats and the $U/K$-system is detailed in section \ref{sec:UK}, and we explain its generalisation to conformally flat $n$-immersions in $S^{2n-2+k}$ in section \ref{sec:codim}. We give a discussion of the dressing transformations of a negative loop on the space of solutions to the $U/K$-system and their associated conformally flat immersions; certain dressing transforms are shown to give rise to geometric Ribaucour transforms in section \ref{sec:rib}. In the final section, we show that solutions of the $U/K$-system defined by a non-semisimple maximal abelian subalgebra give rise to the channel immersions.

\section{Flat lifts, curvature spaces and co-ordinates}\label{sec:flatlift}

In this section we give definitions and explain, via the light-cone model, the correspondence between conformally flat $n$-dimensional immersions in $S^{2n-2}$ and  flat $n$-dimensional immersions in $\cL^{2n-1,1}$.  We also show the existence of line of curvature co-ordinates for conformally flat $n$-immersions in $S^{2n-2}$ with uniform multiplicity one.

\subsection*{The light-cone model}

The light-cone model of the conformal $m$-sphere is now well-understood, its prime advantage being that it \emph{linearises} conformal geometry in $S^m$: the fundamental objects of the theory, subspheres $S^k\subset S^m$ and their intersections, become the geometry of the Grassmannians $G_{m-k}^+(\R^{m+1,1})$ of definite signature planes. Hertrich-Jeromin's book \cite{Hertrich-Jeromin2003} contains an excellent introduction to this, as does Burstall's discussion of isothermic surfaces \cite{Burstall2004}.

Let $\I_{m+1,1}$ denote the diagonal $(m+2)\times(m+2)$ matrix $\diag(1,\ldots,1,-1)$ and $(x,y)=x^t\I_{m+1, 1}y$ the Lorentzian bilinear form on $\R^{m+1,1}$.
\[\cL^{m+1,1}=\{x\in \R^{m+1, 1}\n (x,x)=0\}\]
is the \emph{light-cone} of isotropic vectors in $\R^{m+1,1}$. Fix a choice of unit time-like vector $t_0$. The restriction of $(\ ,\ )$ to $t_0^\perp$ is positive definite, and hence $t_0^\perp$ is isometric to the Euclidean  $\R^{m+1}$. Let $S^m$ denote the set of unit vectors in $t_0^\perp$, then the map\footnote{Throughout we shall use $\ip{\ }$ for the span of a collection of vectors, usually dropping the brackets when referring to the perpendicular space to the span of a vector.}
\[t_0^\perp\supset S^m \to \cL^{m+1,1}:x \mapsto x+t_0\]
is clearly an isometry which, since each isotropic line $\ell\le\cL^{m+1,1}$ intersects the plane $t_0^\perp+t_0$ exactly once, diffeomorphically puts a metric on the projective light-cone $\pr(\cL^{m+1,1})$. However, any other choice of unit time-like $t'_0$ gives a different diffeomorphism and induces a different metric. Indeed the following compound map is seen to be a conformal diffeomorphism from one $m$-sphere to another:
\[\text{\phantom{$t_0^\perp\supset S^m$}}\xymatrix @C0pt @R3pt{\mathllap{t_0^\perp\supset S^m} & \cong & \pr(\cL^{m+1,1}) & \cong & \mathrlap{S^m \subset {t'_0}^\perp}\\
\mathllap{x} \ar@{|->}[rr] & & \ip{x+t_0} \ar@{|->}[rr] &  & \mathrlap{-\frac{x+t_0}{(x+t_0,t'_0)}-t'_0}\\
\mathllap{\D x^2} \ar@{|->}[rrrr] &&&& \mathrlap{\frac{1}{(x+t_0,t'_0)^2}\D x^2} }\text{\phantom{$\frac{x+t_0}{(x+t_0,t'_0)}-t'_0$}}\]
For this reason the projective light-cone is known as the \emph{conformal $m$-sphere}.

\begin{center}
\includegraphics[scale=0.80]{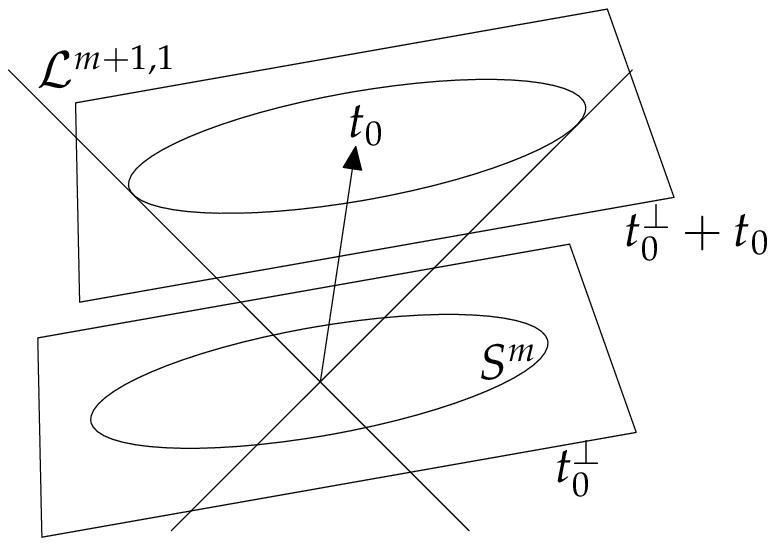}
\end{center}

All geometric properties of $S^m$ that are genuinely conformal are detectable directly in the light-cone and do not depend on the choice of $t_0$, as can be seen by the following theorem.

\begin{thm}[Liouville]
The action of $\rO(m+1,1)$ on the light-cone, and thus on any choice of $m$-sphere $S^m\subset t_0^\perp$, is said to be by \emph{M{\"o}bius transformations}. For $n\ge 3$ \emph{all} (even local!) conformal diffeomorphisms of $S^m$ are M{\"o}bius.
\end{thm}

Liouville's theorem allows us to treat submanifolds of the conformal sphere similarly to those in metric geometry: the existence of submanifolds \emph{up to isometry} is replaced by \emph{up to M{\"o}bius transforms}; in the light-cone picture these really are isometries. In particular, starting from submanifolds in $\cL^{m+1,1}$, we need not worry about specific choices of $t_0$, since examining the submanifold via any other choice merely amounts to a M{\"o}bius transform. We will, however, tend to assume that a fixed choice has been made, if only so that we may anchor discussions in a genuine $S^m$.

\begin{defn}
Given a map $f:M\to S^m\subset t_0^\perp$, a \emph{lift} of $f$ is any map $F:M\to\cL^{m+1,1}$ such that $f+t_0\in\ip{F}$. A \emph{flat lift} is a lift such that the induced metric $\nm{\D F}^2$ on $M$ is flat.
\end{defn}

\begin{rems}
To facilitate computations, we set up notations for moving frames of flat lifts $F:M^n\to \R^{2n-1,1}$. Suppose that $g=(\vect{e}{2n})$ is an $\rO(2n-1,1)$ frame on $M$ such that $\vect{e}{n}$ are tangent to $F(M)$. Let $\lst{\omega}{n}$ be the 1-forms dual to $\vect{e}{n}$. Then $\D F=\sum_{i=1}^n \omega_i \V e_i$. Write
\[\D\V e_A= \sum_{B=1}^{2n} \omega_{BA}\V e_B.\]
Then $g^{-1}\D g= (\omega_{AB})$ and
\begin{gather*}
\omega_{BA}\e_B + \e_A \omega_{AB}=0, \quad {\rm where\,\, } \I_{2n-1, 1}= \diag(\lst{\e}{2n}),\\
\D\omega_{AB}= -\sum_{C=1}^{2n}\omega_{AC}\wedge\omega_{CB}.
\end{gather*}
The shape operator is
$A_{\V e_\a}=-\pi_{\ip{\D F}}(\D\V e_\a)$, the tangential component.  The two fundamental forms and the normal connection are:
\begin{gather*}
\I_F= \sum_{i=1}^n \omega_i^2,\qquad
\II_F=\sum_{i=1}^n \sum_{\a=n+1}^{2n} \omega_i \omega_{\a, i}\V e_\a,\\
\nabla^\perp\V e_\a = \sum_{\beta=n+1}^{2n}\omega_{\beta\a}\V e_\beta, \quad n+1\leq \a\leq 2n.
\end{gather*}
\end{rems}

We may now state the main correspondence of the paper (due to Hertrich-Jeromin \cite{Hertrich-Jeromin1996} when $n=3$).

\begin{thm}
A conformally flat immersion $f:M^n\to S^{2n-2}$ with flat normal bundle has a flat lift $F:M^n\to\cL^{2n-1,1}$ with flat normal bundle.\\
Conversely, any immersed $F:M^n\to\cL^{2n-1,1}$ with a flat, definite signature, metric $\nm{\D F}^2$ and flat normal bundle is a flat lift of a conformally flat immersion $f:M^n\to S^{2n-2}$ with flat normal bundle.
\end{thm}

\begin{rem}
For reference we collect some important relations. If $e^{2u}\I_f$ is flat, we will define, in the proof, a flat lift $F$ as below; similarly, given a flat lift $F$ and a choice of $t_0$, we recover a conformally flat $f$. Everything is related by the following notation:
\[F=-e^u(f+t_0),\qquad f=-\frac{F}{(F,t_0)}-t_0,\qquad e^u=(F,t_0).\]
Since $\pm F$ project to the same $f$, we can always assume that $(F,t_0)$ is positive.

The locations of the normal bundles viewed as subbundles of the trivial bundle are as follows, where both decompositions are orthogonal:
\[M^n\times\R^{2n-1,1}=\ip{\D f}\oplus\ip{f}\oplus\ip{t_0}\oplus N_f=\ip{\D F}\oplus N_F.\]
It is important to note that $N_f\subset N_F\cap F^\perp$.\\
The metric $\nm{\D F}^2$ will almost always be definite: the only time this doesn't happen is if some tangent direction is parallel to $F$ itself (anything else forces a contradiction by requiring either a 2-dimensional isotropic or time-like subspace of $\cL^{2n-1,1}$). It follows that, for us, the \emph{metric} on the normal bundle $N_F$ is always non-degenerate and of signature $(n-1,1)$.
\end{rem}

\begin{proof}
\begin{enumerate}
  \item Suppose that $e^{2u}\I_f$ is flat and define $F:=-e^u(f+t_0)$. Then $\D F=\D u F-e^u\D f$, and thus $F$ is an immersion. Moreover $\I_F=\nm{\D F}^2=e^{2u}\I_f$ is flat, hence $F$ is a flat lift, with positive definite tangent bundle. As is well-known, the bundle $\ip{\D F}$ is then flat (the induced connection is $F$-related to the Levi-Civita of $\I_F$ on $M$).

Now let $\at{\V e_i}{i=1}^n$ be parallel orthonormal sections of $\ip{\D F}$ and $\at{\V n_j}{j=1}^{n-2}$ parallel orthonormal sections of the normal bundle $N_f$. Moreover let $\hat F$ be the unique isotropic section of $N_F$ such that $(\hat F,\V n_j)=0$ and $(\hat F,F)=1$. By the flatness of $\I_F$ and $N_f$, there exist 1-forms such that the moving frame equations are
\[\D
\begin{pmatrix}
 \V e_i&\V n_j&F&\hat F
\end{pmatrix}=
\begin{pmatrix}
 \V e_k&\V n_l &F &\hat F
\end{pmatrix}
\begin{pmatrix}
 0 & \omega_{kj} &\omega_k & \hat\omega_k\\
 \omega_{li} & 0 & 0 & -\eta_l\\
 -\hat\omega_i & \eta_j & 0 & 0\\
 -\omega_i & 0 & 0 & 0
\end{pmatrix}.\]
Here $\I_F=\sum_i\omega_i^2$. The Maurer--Cartan equations quickly yield $\D\omega_i=0$ and $\D\omega_{ij}=\omega_i\wedge\eta_j$, from which we conclude that
\[\omega_i\wedge\D\eta_j=-\D\omega_i\wedge\eta_j=0,\quad\forall i,j.\]
Since the $\omega_i$ form a base of $T^*M$ it follows that the $\eta_j$ are closed and thus the normal bundle $N_F$ is flat.
\item Suppose that an immersed $F$ has flat tangent and normal bundles. Then $f:=-\frac{F}{(F,t_0)}-t_0$ is easily seen to be immersed. Now let $\at{\V N_i}{i=1}^{n-2}$ be any space-like parallel orthonormal sections of $\ip{F,\D F}^\perp$. Defining
\[\V n_i:=\V N_i-\frac{(\V N_i,t_0)}{(F,t_0)}F,\]
it is easy to see that the $\V n_i$ are a parallel orthonormal frame of $N_f$ and hence $f$ has flat normal bundle.\qedhere
\end{enumerate}
\end{proof}

We have now reduced the study of conformally flat $n$-immersions in $S^{2n-2}$ with flat normal bundle to the study of flat $n$-immersions in $\cL^{2n-1,1}$ with flat normal bundle.

\subsection*{Curvature distributions and flat non-degenerate normal bundle}

Suppose that $F$ has flat normal bundle and that the \emph{curvature distributions} of $F$ have constant rank: that is the tangent bundle of $M$ decomposes orthogonally and smoothly as
\[TM=\bigoplus_{i=1}^p E_i,\]
where each $E_i(x)$ is an eigenspace of all the shape operators $A_v(x)$, and the ranks of the $E_i$ are constants. We say that $F$ has \emph{multiplicity}\footnote{$F$ thus has \emph{uniform multiplicity one} if $p=n$ and $m_i=1$ for all $i$.} $(m_1,\ldots,m_p)$ if each $\rank E_i=m_i$. To each curvature distribution there corresponds a \emph{curvature normal} $v_i\in\Gamma N_F$ such that $\at{A_v}{E_i}=(v,v_i)\Id$ (see e.g.\ \cite{Terng1987}). Suppose that $(\vect{e}{n})$ is an $\rO(n)$ tangent frame consisting of principal curvature directions and $(\V e_{n+1}, \ldots,\V e_{2n})$ a parallel $\rO(n-1,1)$ normal frame. Moreover let $\omega_{i\a}= \l_{i\a}\omega_i$ for $1\leq i\leq n$ and $n+1\leq \a\leq 2n$. Then 
\[v_i =\sum_{\a=n+1}^{2n} \e_\a \l_{i\a}\V e_\a.\]

\begin{lemm}\label{lemm:Fcurv}
If $F$ is an immersed flat lift whose principal curvatures along any parallel normal fields have constant multiplicities, then its curvature normals are mutually orthogonal. If the curvature distribution $E_i$ has rank $\ge 2$, then the corresponding curvature normal $v_i$ is isotropic.
\end{lemm}

\begin{proof}
Since $F$ is flat, the Gauss equation implies that 
\[\sum_{\a=n+1}^{2n} \omega_{i\a}\wedge \omega_{\a j}= (v_i, v_j) \omega_i\wedge \omega_j =0,\]
which yields $(v_i,v_j)=0$. Distinct curvature normals are therefore orthogonal, and repeated curvature normals are necessarily isotropic.
\end{proof}

It is now an (almost) obvious corollary that there can only be one $\rank\ge 2$ distribution. A full discussion will wait until section \ref{sec:channel}. Note that a flat immersion $F$ with flat normal bundle is non-degenerate iff all curvature distributions have rank 1 (uniform multiplicity one) \emph{and} all curvature normals are non-zero: theorem \ref{thm:Ffcurvs} will show that this second condition is vacuous.

\subsection*{Non-degeneracy and line of curvature co-ordinates}

The following theorem describes how immersions with non-degenerate normal bundle come equipped with line of curvature co-ordinates.

\begin{thm}\label{thm:coords}
Let $F:M^n\to\cL^{2n-1,1}$ be flat with flat non-degenerate normal bundle and curvature normals $v_1,\ldots,v_n$. Then there exist curvature line co-ordinates $x_1,\ldots,x_n$ on $M$. Moreover $F$ and $\I_F$ can be written entirely in terms of the curvature normals as
\begin{gather}\begin{gathered}\label{eq:Fexplicit}
F=-\sum_{j=1}^n\frac{v_j}{(v_j, v_j)},\\
\I_F= \sum_{i=1}^n \frac{\D x_i^2}{|(v_i, v_i)|}.
\end{gathered}\end{gather}
\end{thm}

\begin{proof}
First recall from lemma \ref{lemm:Fcurv} that the curvature normals are non-isotropic. Since
\begin{gather}\label{eq:cross1}
\pi_{N_F}\D\V e_i=\omega_iv_i,
\end{gather}
it follows that
\begin{gather}\label{eq:cross2}
\pi_{\ip{\D F}}\D v_j=-\omega_j(v_j, v_j)\V e_j.
\end{gather}
Since the $v_j$ are a frame of $N_F$, there exist functions $\mu_j$ such that $F=\sum\mu_jv_j$. However $\D F=\sum\omega_j\V e_j$. Putting this together quickly gives $\mu_j=-(v_j,v_j)^{-1}$.

Assume $\epsilon_i= {\rm sgn}(v_i, v_i)$. Let $n_i>0$ be defined by
\begin{gather}\label{eq:m}
n_i^2= \epsilon_i (v_i,v_i)=\nm{(v_i, v_i)}.
\end{gather}
Now $n_i\omega_i=\left(\epsilon_i n_i^{-1}v_i,\D\V e_i\right)$. By \eqref{eq:cross1},\eqref{eq:cross2} and the fact that $\V e_i$ and $n_i^{-1}v_i$ are unit length, it is immediate that\footnote{$(p\,\overset\wedge ,\,q)(X,Y)=(p(X),q(Y))-(p(Y),q(X))$.}
\[\D(n_i\omega_i)=\epsilon_i\left(\D(n_i^{-1}v_i)\,\overset\wedge , \,\D\V e_i\right)=0.\]
The $\omega_i$ are independent, thus we have co-ordinates $x_i$. Since the $\omega_i$ diagonalise the fundamental forms of $F$, it follows that the $\D x_i$ do similarly and are thus curvature line co-ordinates. The expression for $\I_F$ is then immediate.
\end{proof}

\subsection*{Comparing curvatures}

Since $\D N_f\perp\ip f$ and $\D f\perp N_f$, it follows that if $\V n_i$ are parallel sections of $N_f$, then $\{f,\V n_i\}$ are a parallel frame for the Euclidean normal bundle $\ip f\oplus N_f\subset M\times t_0^\perp$. $f$ therefore has flat normal bundle when viewed as an immersion into $\R^{2n-1}=t_0^\perp$. It also therefore has an orthogonal curvature distribution; moreover $S^{2n-2}$ will inherit a curvature distribution from $\R^{2n-1}$. \emph{A priori} we now have \emph{three} curvature distributions on $M$: the next theorem tidies things up. We also have curvature normals for $F$ and $f$ which we would like to relate.

\begin{thm}\label{thm:Ffcurvs}
Let $f:M^n\to S^{2n-2}$ be a conformally flat immersion with flat normal bundle, and $F=-e^{-u}(f+t_0)$ flat lift. Then
\begin{enumerate}
\item The curvature directions (on $M^n$) induced by $F$ and $f$ are identical.
\item If $v_i,v_i^S,v_i^\R$ are corresponding curvature normals of $F$ and $f$ (the latter as a map into $S^{2n-2}$ and $\R^{2n-1}$ respectively), and $\V e_i,\tilde{\V e}_i$ the corresponding curvature directions, then we have the relations
\begin{equation}\label{eq:Ffcurvs}
\begin{gathered}
\tilde{\V e}_i=-\V e_i+e^{-u}(\V e_i,t_0)F\\
v_i^S=-e^u\pi_{N_f}v_i=v_i^\R+f.
\end{gathered}
\end{equation}
\item All curvature normals of $F$ are non-zero.
\item $F$ has non-degenerate normal bundle iff $f$ has uniform multiplicity one.
\end{enumerate}
\end{thm}

\begin{proof}
Let $\V e_i$ be orthonormal curvature directions for $F$ and define $\tilde{\V e}_i$ as in the theorem. The $\tilde{\V e}_i$ are orthonormal, and we calculate that
\[e^u\D f=\D uF-\D F=\sum_i\omega_i\tilde{\V e}_i-e^{-u}(\V e_i,t_0)\omega_iF+\D u F=\sum_i\omega_i\tilde{\V e}_i,\]
hence the $\tilde{\V e}_i$ frame $\ip{\D f}$. Moreover
\begin{gather*}
\pi_{N_f}\D\tilde{\V e}_i=-\pi_{N_f}\D\V e_i=-\omega_i\pi_{N_f}v_i,\\
\pi_f\D\tilde{\V e}_i=(\D\tilde{\V e}_i,f)f=-(\tilde{\V e}_i,\D f)f=-e^{-u}\omega_if,
\end{gather*}
hence the $\tilde{\V e}_i$ are curvature directions for $f$. It follows that the curvature normals for $f$ as a map into $S^{2n-2}$ are $v_i^S=-e^{-u}\pi_{N_f}v_i$. We thus have (1) and (2).

For (3) and (4) we calculate explicitly:
\begin{equation}\label{eq:viproj}
\begin{split}
\pi_{N_f}v_i&=v_i-\sum_{j=1}^n(v_i,\tilde{\V e}_j)\tilde{\V e}_j-(v_i,f)f+(v_i,t_0)t_0\\
&=v_i-\sum_{j=1}^n(v_i,e^{-u}(\V e_j,t_0)F)(-\V e_j+e^{-u}(\V e_j,t_0)F)\\
&\qquad\qquad\qquad-(v_i,e^{-u}F+t_0)(e^{-u}F+t_0)+(v_i,t_0)t_0\\
&=v_i-e^{-u}\sum_{j=1}^n(\V e_j,t_0)\V e_j+e^{-u}t_0\\
&\qquad\qquad\qquad+e^{-2u}\left(\sum_{j=1}^n(\V e_j,t_0)^2+1-e^u(v_i,t_0)\right) F.
\end{split}
\end{equation}
Supposing that $v_i=0$, we readily obtain
\[t_0=\sum_{j=1}^n(\V e_j,t_0)\V e_j-e^{-u}\left(\sum_{j=1}^n(\V e_j,t_0)^2+1\right) F.\]
However taking the norm squared of both sides results in the contradiction
\[-1=\sum_{j=1}^n(\V e_j,t_0)^2.\]
It follows that $v_i$ cannot be zero, yielding (3).

By \eqref{eq:viproj} we have that
\begin{gather}\label{eq:vidist}
v_i^S-v_j^S=v_i-v_j-e^{-u}(v_i-v_j,t_0)F.
\end{gather}
If the $v_i^S$ are distinct, then the $v_i$ are distinct and, by lemma \ref{lemm:Fcurv}, orthogonal. Rearranging and taking the norm squared of \eqref{eq:vidist} yields $\nm{v_i-v_j}>0$, hence no pair of $v_i$ are parallel isotropic. The $v_i$ are thus all non-isotropic and $N_F$ is non-degenerate. Conversely, if $v_i^S=v_j^S$, then $v_i-v_j\in\ip{F}$, so that either $v_i=v_j$ is a repeated curvature of $F$, or $F\in\ip{v_i-v_j}$. The latter case yields a contradiction, for it would require that $\D F\in\ip{\omega_i,\omega_j}$ only. This establishes (4).
\end{proof}

\begin{rems}
Observe that if $f$ has distinct curvature normals, then $N_f$ is automatically non-degenerate: $\ip{v^S_1,\ldots,v^S_n}=\pi_{N_f}\ip{\lst{v}{n}}$. Indeed the proof shows that $v_i^S-v_j^S\iff v_i=v_j$.\\
All curvature normals of $F$ are non-zero, hence the only way $N_F$ can be degenerate is if $F$ has a repeated non-zero curvature, necessarily isotropic. Thus $F$ is non-degenerate iff it has uniform multiplicity one. By contrast, $f$ may have a zero curvature normal --- i.e.\ the curvature line is a geodesic --- whether $F$ is degenerate or not. The multiplicities of $f$ are however identical to those of any flat lift.\\
When $F$ is non-degenerate we may obtain further relations from \eqref{eq:Fexplicit} and \eqref{eq:Ffcurvs}:
\begin{gather*}
\I_f=e^{-2u}\I_F=e^{-2u}\sum_i\frac{\D x_i^2}{|(v_i,v_i)|},\\
\begin{aligned}
\II_f^S&=e^{-2u}\sum_i\frac{\D x_i^2}{|(v_i,v_i)|}v_i^S=\II_f^\R+\I_f\cdot f\\
&=-e^{-u}\pi_{N_f}\II_F=-e^{-u}\sum_i\frac{\D x_i^2}{|(v_i,v_i)|}\pi_{N_f}v_i.
\end{aligned}
\end{gather*}
\end{rems}

We will say more on the curvature distributions of conformal flats with repeated curvature normals in section \ref{sec:channel}.

\section{The \texorpdfstring{$U/K$}{U/K}-system}\label{sec:UK}

We prove that the Gauss-Codazzi equation for flat $n$-immersions in $\cL^{2n-1,1}$ with flat normal bundle and "good" co-ordinates is the $\otwonone$-system.

The $U/K$-system \cite{Terng1997} is a general construction common to any symmetric space. We give its construction for the case $U=\rO(2n-1,1)$ and $K=\rO(n)\times\rO(n-1,1)$: these will be viewed as matrices where $g\in U\iff g^TI_{2n-1,1}g=I_{2n-1,1}$. $K$ is then represented by the block $n\times n$ diagonal matrices. The Lie algebras of $U$ and $K$ will be written $\fu$ and $\fk$, while the Killing perp of $\fk$ will be denoted by $\fp$. As is well-known, $U/K$ is a pseudo-Riemannian symmetric space defined by the involution $\tau(\xi)= \I_{n,n}\xi\I_{n,n}^{-1}$ with $\pm$-eigenspaces
\[ \fk= \fo(n) \times \fo(n-1, 1), \quad \fp=\left\{\bpm 0& \xi^T J\\ -\xi & 0\epm\,\big|\, \xi\in \fgl(n)\right\}, \qquad J=\I_{n-1, 1}.\]
Moreover, 
\[[\fk,\fk],\ [\fp,\fp]\subset\fk,\qquad [\fk,\fp]\subset\fp.\]
 The conjugation of $K$ on $\fp$ is 
\[\bpm g_1& 0\\ 0& g_2\epm\ast \bpm 0&\xi^TJ\\ -\xi &0\epm=\bpm 0& g_1 \xi^TJ g_2^{-1}\\ -g_2\xi g_1^{-1} &0\epm.\]

Let $\fa$ be a maximal abelian subalgebra in $\fp$, and $\lst{a}{n}$ a basis of $\fa$. The $U/K$-system defined by $\fa$ is the PDE \eqref{eq:uk} for $\Xi:\R^n\to \fa^\perp\cap \fp$. It is known and can be easily checked that the following statements are equivalent:
\begin{enumerate}
\item $\Xi$ is a solution of the $U/K$-system,
\item The following one-parameter family of $\fo(2n-1, 1,\C)$ connection 1-forms is flat for all $\l\in \C$:
\begin{equation}\label{eq:bb}
\theta_\l= \sum_{i=1}^n (\l a_i+ [a_i, \Xi]) \D x_i,
\end{equation}
\item $\theta_\lambda=\sum_{i=1}^n (\lambda a_i + [a_i, \Xi])\D x_i$ is flat for some $\lambda\in\R\setminus\{0\}$.
\end{enumerate} 
We call $\theta_\l$ the \emph{Lax pair} of the $U/K$-system, and a solution $\Phi_\l:\R^n\times\C\to\rO(2n-1,1,\C))$ to $\Phi_\l^{-1} \D \Phi_\l= \theta_\l$ satisfying the {\it $U/K$-reality condition},
\begin{equation}\label{eq:real}
 \cl{\Phi_\lambda}=\Phi_{\cl\lambda},\quad \tau\Phi_\lambda=\Phi_{-\lambda},
 \end{equation}
an \emph{extended flat frame} for the solution $\Xi$ of the $U/K$-system. Observe that:
\begin{enumerate}
\item[(a)] We could normalise the extended flat frame at a base point, e.g.\ $\Phi_\lambda(0)=\Id$, which will be called the \emph{normalised extended flat frame}. Choosing a different extended flat frame merely affects submanifold geometry by rigid motions (equivalently M{\"o}bius transforms of $S^{2n-2}$).
\item[(b)] If $\Phi_\l$ is an extended flat frame, then $\Phi_r\in \rO(2n-1,1)$ for all $r\in \R$ and $\Phi_0\in \rO(n)\times \rO(n-1,1)$.
\end{enumerate}

There are both semisimple and non-semisimple maximal abelian subalgebras in $\fp$. These comprise $n$ conjugacy classes, each of which give rise to non-equivalent $U/K$-systems. We will show in this section that the $U/K$-system defined by a semisimple abelian subalgebra in $\fp$ is the Gauss-Codazzi equation for flat lifts of conformally flat immersions with uniform multiplicity one. In section \ref{sec:channel} we will see that solutions of a $U/K$-system defined by a non-semisimple maximal abelian subalgebra give rise to conformally flat immersions with one multiplicity greater than one.

\begin{thm}\label{thm:ndimconf}
Suppose that $F:M^n\to \cL^{2n-1,1}$ is a flat immersion with flat non-degenerate normal bundle, induced metric $\I_F=\sum_{i=1}^nh_i^2\D x_i^2$, and that $(x_1, \ldots, x_n)$ is a line of curvature co-ordinate system. Let $\Psi=(\vect{e}{n},\vect{u}{n})$ be an $\rO(2n-1, 1)$-frame such that the $\V e_i$ are principal curvature directions and the $\V u_i$ are parallel to the curvature normals of $F$. Let $\frac UK=\frac{\rO(2n-1,1)}{(\rO(n)\times\rO(n-1,1))}$, $\fu=\fk+\fp$ its Cartan decomposition, and
\[\fa=\left\{ \bpm 0& DJ\\ -D &0\epm\,\bigg|\, D\, {\rm diagonal}\right\},\]
a Cartan subalgebra in $\fp$, where $J=\I_{n-1,1}$. Set $\xi=(\xi_{ij})$, where
\[\xi_{ij}=\begin{cases}
           \frac{(h_i)_{x_j}}{h_j},&i\neq j,\\
           0&i=j.
           \end{cases}\]
Then $\Xi:=\begin{pmatrix}0&\xi^T\\-J\xi&0\end{pmatrix}$ is a solution of the $U/K$-system defined by $\fa$, 
\begin{equation}\label{eq:uk2}
[a_i, \Xi_{x_j}] -[a_j, \Xi_{x_i}] + [[a_i, \Xi], [a_j, \Xi]]=0, \quad i\not=j,
\end{equation}
and
\[\Psi^{-1}\D\Psi= \sum_{i=1}^n (a_i + [a_i, \Xi]) \, \D x_i,\] 
where $a_i= \bpm 0& e_{ii} J\\ -e_{ii} &0\epm$ for $1\leq i\leq n$. Moreover, let $g=\bpm g_1&0\\ 0& g_2\epm$ in $K$ be a solution of $g^{-1} \D g=\sum_{i=1}^n [a_i, \Xi] \D x_i$, then there exists a constant null vector $\V c\in \R^{n-1,1}$ such that $F=\Psi\bpm 0\\ g_2^{-1}\V c\epm$. 
\end{thm}

\begin{proof}
Let $v_1, \ldots, v_n$ be the curvature normals for the flat lift $F:M^n\to\cL^{2n-1,1}$. We may assume that $(v_n,v_n)<0$. Define (as in \eqref{eq:m})
\[\epsilon_i = {\rm sgn} (v_i, v_i), \quad n_i= |\epsilon_i(v_i, v_i)|^{1/2},\quad \text{then } v_i = n_i\V u_i.\]
$(\vect{u}{n})$ is an $\rO(n-1,1)$ normal frame. Let $(\vect{e}{n})$ be an $\rO(n)$ tangent frame consisting of corresponding principal curvature directions. 

It is clear that
\begin{gather*}
\theta:=\Psi^{-1}\D\Psi=\begin{pmatrix}
                          A&-\delta J\\\delta &B
                        \end{pmatrix}
\end{gather*}
where $J=\I_{n-1,1}$, $\delta=\diag(\D x_1,\ldots,\D x_n)$ and the Lie algebra valued 1-forms $A,B$ are the connection forms for the induced flat connections on $\ip{\D F}$ and $N_F$. The flatness $\D\theta+\theta\wedge\theta=0$ reads
\begin{gather}\label{eq:thetaflat}
\D A+A\wedge A=0=\D B+B\wedge B=\delta\wedge A+B\wedge\delta.
\end{gather}
By considering where the $\D x_i\wedge \D x_j$ terms appear in the third equation, it is not hard to see that $A_{ij},B_{ij}\in\ip{\D x_i,\D x_j}$ and that there therefore exists a unique map
\[\xi:M^n\to\{\text{off-diagonal $n\times n$ matrices}\},\]
such that
\[A=\delta\xi-\xi^T\delta,\qquad B=\delta \xi^T-J\xi\delta J.\]
The first two equations of \eqref{eq:thetaflat} constitute a system of PDEs in the entries of $\xi$. Indeed writing $\Xi=\begin{pmatrix}
   0&\xi^T\\-J\xi&0
\end{pmatrix}$ and $\sum_ia_i\D x_i=\begin{pmatrix}
   0&\delta J\\-\delta &0
\end{pmatrix}$, we see that $[\Xi,a_i\D x_i]=\diag(A,B)$. It is straightforward to see that $\Xi$ is a map $\Xi:M^n\to\fp\cap\fa^\perp=[\fk,\fa]$, and is thus a solution of the $U/K$-system \eqref{eq:uk2}. 

Since $N_F$ is flat, and the normal connection 1-form defined by the normal frame $(\vect{u}{n})$ is $B$, there exists $g_2:M\to \rO(n-1,1)$ such that $g_2^{-1}\D g_2= B$ and 
\[(\V e_{n+1},\ldots,\V e_{2n})=(\vect{u}{n})g_2^{-1}\]
is a parallel normal frame. Since $F$ is a null parallel normal section, there exists a constant null vector $\V c$ such that
\[F= (\V e_{n+1}, \ldots,\V e_{2n}) \V c = (\vect{u}{n}) g_2^{-1} \V c= \Phi\twovec{0}{g_2^{-1}\V c}.\]
\end{proof} 

\begin{thm}\label{thm:ndimconf2}
 Let $\Xi$ be a solution to the $U/K$-system \eqref{eq:uk2}, and $\Phi_\l$ an extended flat frame for the Lax pair $\theta_\l= \sum_{i=1}^n (\l a_i+[a_i, \Xi]) \D x_i$. Then:
\begin{enumerate}
 \item $\Phi_0$ is of the form $\bpm g_1 & 0\\ 0& g_2\epm$, and if $\V c\in \R^{n-1,1}$ is a constant non-zero null vector, then $F_\bc:= \Phi_1\stwovec 0 {g_2^{-1}\V c}$ is (locally) a flat lift of a conformally flat immersion into $S^{2n-2}$.
 \item $F_\bc$, its fundamental forms and curvature normals $v_i$ are given explicitly by
\begin{gather*} 
F_\bc= \sum_{i=1}^n q_i\V u_i, \qquad \I_{F_\bc}= \sum_{i=1}^n q_i \D x_i^2,\\
\II_{F_\bc}= \sum_{i, j=1}^n q_i \D x_i^2\V u_i, \qquad v_i = q_i^{-1}\V u_i,
\end{gather*}
where $\Phi_1=(\V e_1, \ldots \V e_n,\V u_1, \ldots, \V u_n)$ and $g_2^{-1}\V c= (q_1, \ldots, q_n)^T$.
\item If $\bb$ is another non-zero null vector in $\R^{n-1,1}$, then $F_\bc(x)\mapsto F_\bb(x)$ is a Combescure transform and is a Christoffel transform if $h(g_2^{-1}\bb)$ and $h(g_2^{-1}\bc)$ have opposite sign, where $h(y)= \prod_{i=1}^n y_i$.   
\end{enumerate}
\end{thm}

\begin{proof}
Such an $F_\bc$ is clearly a flat lift.  

Write $\theta_1= \bpm A & -\d J \\ \d & B \epm$. Then $g_2^{-1}\rd g_2= B$. 
It follows from $\Phi^{-1}_1\rd \Phi_1 = \theta_1$ and $g_2^{-1}\rd g_2= B$  that
\[\rd F_\bc= \Phi_1\bpm -\d J g_2^{-1} \bc\\ 0\epm.\]
Write $\Phi_1= (\vect{e}{n}, \vect{u}{n})$ and $g_2^{-1}\bc= (q_1, \ldots, q_n)^t$.  Then 
\[\rd F_\bc= -\sum_{j=1}^n \e_j q_j e_j.\]
Statements (2) and (3) follow.
\end{proof}

We recall below three known approaches to constructing solutions to the $U/K$-system (see e.g.\ \cite{Terng1997, Terng2000, Terng2005}).

\begin{description}
  \item[Dressing] Given a solution we may apply dressing actions to find new solutions to \eqref{eq:uk2}. We shall pursue this in section \ref{sec:rib}.
  \item[Cartan--K\"ahler] \cite{Terng2005} The $U/K$-system is unchanged with respect to the substitution $(a_i,x_i)\mapsto (b_j,y_j)$ where $b_j=\sum_ip_{ij}a_i,\ y_j=\sum_ip^{ji}x_i$ for any constant invertible $(p_{ij})$, where $(p_{ij})^{-1}= (p^{ij})$. Almost any choice of $(p_{ij})$ will result in $b_1$ being \emph{regular}, i.e. $\ad b_1:\fp\cap\fa^\perp\to\fk$ is injective. 
The corresponding exterior differential system is involutive, hence we may apply the Cartan--K\"ahler theorem: given local analytic initial data $\Xi_0:(-\ve,\ve)\to\fp\cap\fa^\perp$, there is a unique local analytic solution $\Xi$ to \eqref{eq:uk2} which satisfies $\Xi(y_1,0,0)=\Xi_0(y_1)$. Note that $\dim(\fa^\perp\cap \fp)= n^2-n$.
  \item[Inverse scattering] \cite{Terng1997} If $b_1\in\fa$ is regular, then given rapidly decreasing initial data $\Xi_0:\R\to\fp\cap\fa^\perp$ with small $L^1$ norm, there exists a unique global smooth solution $\Xi$ to the $U/K$-system \eqref{eq:uk2} such that $\Xi(y_1,0,\ldots,0)=\Xi_0(y_1)$.
\end{description}

We will compute the dressing action of the simplest type rational loops in section \ref{sec:rib}

\begin{rem}
The discussion on constructions of solutions of the $U/K$-system given above and in Theorem \ref{thm:ndimconf2} imply that local flat $n$-immersions in $\cL^{2n-1,1}$ with flat, non-degenerate normal bundle and uniform multiplicity one are determined by $n^2-n$ functions of one variable (the restriction of $\xi$ to a regular line). We state this more precisely. Fix a regular element $b=\sum_{i=1}^n b_i a_i$ in $\fa$. Given a real analytic map $\xi_0=((\xi_0)_{ij}):(-\e, \e)\to\fgl(n)$ such that $(\xi_0)_{ii}=0$ for all $1\leq i\leq n$, there is a unique solution $\Xi$ of the $U/K$-system \eqref{eq:uk2} defined on an open subset of the origin of $\R^n$ such that $\Xi(tb_1, \ldots, tb_n)=\bpm 0 & \xi_0(t)^T\\ -J\xi_0(t) & 0\epm$. The same statement holds if we replace $\xi_0$ by a rapidly decaying smooth map on $\R$ whose $L^1$ norm is less than $1$. For each solution of the $U/K$-system, Theorem \ref{thm:ndimconf2} gives a family of flat $n$-immersions in $\cL^{2n-1,1}$ parameterised by the light-cone of $\R^{n-1,1}$.

\end{rem}

\section{Conformally flat $n$-immersions in $S^{2n+k-2}$}\label{sec:codim}

An immersion $f:M^n\to S^{2n+k-2}$ with flat normal bundle is said to have \emph{uniform multiplicity one} if it has constant multiplicities and each curvature distribution has rank one. It is easy to see that the proofs of the existence of flat lifts and line of curvature co-ordinates still work for these immersions. Let $\lst{v}{n}$ be the curvature normals for a flat lift $F$, $N_v=\ip{v_i}_{i=1}^n$ the \emph{curvature normal bundle}, and $N_v^\perp$ the orthogonal complement of $N_v$ in $N_F$. Note that Theorem \ref{thm:ndimconf} holds if we replace the $U/K$-system with the $\frac{\rO(2n+k-1, 1)}{\rO(n)\times \rO(n+k-1, 1)}$-system and assume that $N_v^\perp$ is flat. We state the analogous results below. 

\begin{thm}
Let $f:M^n\to S^{2n+k-2}$ be a conformally flat immersion with flat normal bundle and uniform multiplicity one. Then:
\begin{enumerate}
\item There is a flat lift $F:M^n\to \cL^{2n+k-1, 1}$ for $f$ with flat normal bundle and uniform multiplicity one,
\item There exist line of curvature co-ordinates for $F$.
\end{enumerate}
\end{thm}

The symmetric space $\frac{\rO(2n+k-1, 1)}{\rO(n)\times \rO(n+k-1, 1)}$ is defined by the involution $\tau(\xi)=\I_{n, n+k} \xi \I_{n, n+k}^{-1}$, and the $\pm 1$ eigenspaces of $\tau$ are $\fk= \fo(n)\times \fo(n+k-1, 1)$, and
\[\fp=\left\{\bpm 0& \xi \\ -J\xi^t &0\epm\, \big| \, \xi\in \fgl(n, n+k)\right\},\qquad {\rm where\quad} J=\I_{n+k-1,1}.\]
Set 
\begin{gather*}
\fa_+=\ip{a_i}_{i=1}^n, \quad {\rm where\quad} a_i=e_{i, n+i} - e_{n+i, i}, \\
\fa_-=\ip{b_i}_{i=1}^n, \quad {\rm where\quad} b_i= e_{i, n+k+i}- \e_i e_{n+k+i, i},
\end{gather*}
where $\e_i=1$ for $i<n$ and $\e_n=-1$. Then both $\fa_+, \fa_-$ are maximal abelian subalgebras in $\fp$, and 
\begin{gather*}
\fa_+^\perp\cap\fp= \left\{\bpm 0& \xi\\ -J \xi^t&0\epm \,\big| \, \xi=(\xi_{ij})\in \fgl(n, n+k), \xi_{ii}=0, 1\leq i\leq n\right\},\\ 
\fa_-^\perp\cap \fp= \left\{ \bpm 0&\xi\\ -J\xi^t&0\epm\, \big| \, \xi=(\xi_{ij})\in \fgl(n, n+k), \xi_{i, k+i}=0\right\}, 
\end{gather*}

\begin{thm}  Let $F:M^n\to \cL^{2n+k-1, 1}$ be a flat lift of $f$ with flat normal bundle and uniform multiplicity one, $\lst{v}{n}$ the curvature normals for $F$,  and $\fa_\pm$ as above. Suppose that the induced connection on the subbundle $N_v^\perp$ of $N_F$ is flat, where $N_v=\li v_i\ri_{i=1}^n$.

\noindent (i) If $(v_i,v_i)>0$ for all $1\leq i\leq n$,  then there exists a solution $\Xi:M^n\to \fa_+^\perp\cap \fp$ of the $\frac{\rO(2n+k-1, 1)}{\rO(n)\times \rO(n+k-1, 1)}$-system defined by $\fa_+$ and an $\rO(2n+k-1, 1)$ frame $\Phi=(\vect{e}{n},\vect{u}{n+k})$ such that
\begin{enumerate}
\item $\V e_i$ are principal curvature directions, $\V u_i=v_i/(v_i,v_i)^{1/2}$ for $1\leq i\leq n$, and $(\D\V u_i,\V u_j)=0$ if $n< i, j\leq n+k$,
\item $\Phi^{-1}\D \Phi= \sum_{i=1}^n (a_i + [a_i, \Xi])\D x_i$.  
\end{enumerate}
\noindent (ii) If $(v_n, v_n)<0$, then there exists a solution $\Xi:M^n\to \fa_-^\perp\cap \fp$ of the $\frac{\rO(2n+k-1,1)}{\rO(n)\times\rO(n+k-1, 1)}$-system defined by $\fa_-$ and an $\rO(2n+k-1, 1)$ frame $\Phi=(\vect{e}{n},\vect{u}{n+k})$ such that
\begin{enumerate}
\item $\V e_i$ are principal curvature directions, $\V u_{k+i}=v_i/(v_i,v_i)$ for $1\leq i\leq n$, and $(\D\V u_i,\V u_j)=0$ if $1\leq i, j\leq k$,
\item  $\Phi^{-1}\D \Phi= \sum_{i=1}^n (b_i + [b_i, \Xi])\D x_i$. 
\end{enumerate}
Moreover, there is a constant null vector $\V c\in \R^{n+k-1,1}$ such that 
$F=\Phi\stwovec 0{g_2^{-1}\V c}$, where $\sum_{i=1}^n [a_i, \Xi]\D x_i = g^{-1}\D g$ and $g=\bpm g_1&0 \\ 0& g_2\epm$.
\end{thm}

\begin{thm}
Let $\Xi$ be a solution of the $\frac{\rO(2n+k-1,1)}{\rO(n)\times\rO(n+k-1, 1)}$-system defined by $\fa_+$ ($\fa_-$ resp.), $\Phi_\l$ an extended flat frame for the Lax pair of $\Xi$ and $\V c\in \R^{n+k-1,1}$ a null vector. Then $\Phi_0=\bpm g_1&0\\ 0& g_2\epm$, and $F=\Phi_1\stwovec 0{g_2^{-1}\V c}$ is flat with flat normal bundle, uniform multiplicity one, and the subbundle $N_v^\perp$ is flat.\\
A different choice of null vector $\bb$ yields a Combescure transform $F_\bb=\Phi_1\stwovec 0{g_2^{-1}\bb}$ of $F$.
\end{thm}

\section{Dressing action and Ribaucour transformations}\label{sec:rib}

We first give a brief review of the dressing action (cf. \cite{Terng2000}) for the $U/K$-system. Then we construct simple elements in the rational loop group with two poles that lie in the negative loop group, and give explicit formulae for the dressing action of simple elements on solutions of the $U/K$-system and the corresponding action on the flat lifts of conformally flat $n$-immersions in $S^{2n-2}$. We prove in the end of the section that the dressing action of simple elements on flat lifts are Ribaucour transformations enveloping $n$-sphere or $n$-hyperboloid congruences which moreover project down to Ribaucour transforms enveloping $n$-sphere congruences of conformally flat $n$-immersions in $S^{2n-2}$. 

Throughout, $U^\C=\rO(2n-1,1,\C)$ where $\cl{\phantom{\lambda}}$ will denote conjugation across the real form $U\subset U^\C$. $\tau$ will refer both to the symmetric involution of $U^{\C}$ which defines the symmetric space $U/K$ \emph{and} its derivative at the identity, whose eigenspaces are $\fk,\fp$.

Fix $\e>0$. Let $\cO_\e=\{\l\in \pr^1\n \e^{-1}<\l\le\infty\}$, and $L(U)$ the group of holomorphic maps from $\cO_\e\setminus \{\infty\}$ to $U^C$ that satisfy the $U/K$ reality condition \eqref{eq:real}.  Let $L^+(U)$ denote the subgroup of $g_+\in L(U)$ that is the restriction of a holomorphic map on $\C$, and $L^-(U)$ the subgroup of $g_-\in L(U)$ that is the restriction of a holomorphic map defined on $\cO_\e$ such that $g_-(\infty)=\Id$.  The Birkhoff Factorisation Theorem implies that the multiplication maps $L^+(U)\times L^-(U)\to L(U)$ and $L^-(U)\times L^+(U)\to L(U)$ defined by $(g_+, g_-)\mapsto g_+g_-$ and $(g_-, g_+)\mapsto g_-g_+$ are injective and the images are open dense.  Hence the \emph{big cell} 
$${\mathcal O\/}(U):=(L^+(U)L^-(U))\cap (L^-(U)\cap L^+(U))$$ is open and dense.   The local dressing action of $L^-(U)$ on $L^+(U)$ is defined as follows. Since the big cell ${\mathcal O}(U)$ is open and dense, given $g_\pm\in L^\pm(U)$, there generically exist $\hat g_\pm\in L^\pm(U)$ such that $g_-g_+=\hat g_+\hat g_-$. Then $g_-\sharp g_+:= \hat g_+$ defines a local action of $L^-(U)$ on $L^+(U)$: this is the \emph{dressing action}.

It is proved in \cite{Terng2000} that the dressing action of $L^-(U)$ induces an action of $L^-(U)$ on the space of solutions of the $U/K$-system:

\begin{thm} \cite{Terng2000}
Let $\Phi_\l$ be the normalised flat frame of the Lax pair of a solution $\Xi$ of the $U/K$-system, and $g\in L^-(U)$. Then
\begin{enumerate}
\item There exists an open subset $B$ of the origin in $\R^n$ such that the dressing action of $g$ at $\Phi(x)$, $\ti\Phi(x):=g\sharp \Phi(x)$, is defined for all $x\in B$, in other words, we can factor $g\Phi(x)= \ti \Phi(x) \ti g(x)$ with $\ti \Phi(x)\in L^+(U)$ and $\ti g(x)\in L^-(U)$ for all $x\in B$, 
\item  $\ti\Phi_\l$ is the normalised frame for a solution of the $U/K$-system, which will be denoted by $g\sharp \Xi$.
\end{enumerate}
\end{thm}

In general it is difficult to calculate the dressing action of a given $g_-\in L^-(U)$ on $L^+(U)$ explicitly. It is, however, now standard theory that if $g\in L^-(U)$ is rational \cite{Terng2000, Bruck2002, Burstall2004}, then the dressing action of $g$ on $L^+(U)$  can be computed explicitly via residue calculus. Moreover, the action of a rational $g\in L^-(U)$ with minimal number of poles (a so-called \emph{simple} element) often corresponds to known classical transforms (e.g.\ B\"acklund transforms of pseudospherical surfaces, Darboux transforms of isothermic surfaces, etc.). 

Let $\tau$ be conjugation by $\rho=\I_{n,n}$, choose a scalar $\alpha\in\R^\times\cup i\R^\times$, and an isotropic line $\ell$ such that either
\begin{equation}\label{ba}
\ell\le\R^{2n-1,1}\text{ and }\alpha\in\R^\times,\quad\text{or}\quad \ell\le\R^n\oplus i\R^{n-1,1}\text{ and }\alpha\in i\R^\times.
\end{equation}
Let $L=\ell^\C$ and suppose in addition that $\rho L\neq L$ (equivalently $\rho L\not\perp L$). Let $\pi_L$ denote projection onto $L$ away from $\rho L^\perp$. In fact, if $\ell=\ip{v}$ with $|v|=0$ and $(v,\rho v)=1$, then 
\[\pi_L= vv^T\rho, \quad \pi_{\rho L}= \rho vv^T.\]
Define the \emph{simple element} $p_{\alpha,L}$ by
\begin{equation}\label{bd}
p_{\alpha,L}(\lambda)=\frac{\lambda-\alpha}{\lambda+\alpha}\pi_L+\pi_{(L\oplus\rho L)^\perp}+\frac{\lambda+\alpha}{\lambda-\alpha}\pi_{\rho L}.
\end{equation}
Then 
\[p_{\a, L}= g_{\a, \rho L} g_{-\a, L}, \qquad {\rm where}\quad g_{\a, L}(\l)= \I + \frac{2\a}{\l-\a}\pi_L.\]
It is easily checked that $p_{\a,L}$ is an element of $L^-(U)$.

Let $\Xi:\R^n\to\fa^\perp\cap\fp$ be a solution to the $U/K$-system \eqref{eq:uk2}. An extended frame $\Phi_\lambda$ of the Lax pair $\theta_\l$ \eqref{eq:bb} of $\Xi$ is called the \emph{normalised extended frame} if $\Phi_\l(0)= \Id$. Since the solution of an ODE depending on a holomorphic parameter is holomorphic in that parameter, the map $\l\mapsto \Phi_\l(x)$ lies in $L^+(U)$.  

Using the dressing action of $g_{\a,L}$ computed in \cite{Terng2000}, or alternatively from \cite{Bruck2002, Burstall2004}, we get the following:

\begin{thm}\label{thm:dress}
Let $\Phi_\l(x)$ be the normalised extended frame for a solution of the $U/K=\frac{\rO(2n-1,1)}{\rO(n)\times \rO(n-1,1)}$-system, and $p_{\a,L}$ the simple element defined by \eqref{bd}. Then there is an open subset $B$ of the origin in $\R^n$ such that $\ti \Phi(x):= p_{\alpha,L}\# \Phi(x)$, the dressing action of $p_{\a, L}$ on $\Phi(x)$, is defined for all $x\in B$. Moreover, 
\begin{enumerate}
\item $\ti L(x)= \Phi_\a^{-1}(x)L$ and $\a$ satisfy \eqref{ba},
\item 
\[\ti \Phi(x)=p_{\alpha,L}\# \Phi(x)=p_{\alpha,L} \Phi(x) p^{-1}_{\alpha,\Phi^{-1}_\alpha(x) L}\]
 is the normalised extended frame of a new solution of the $U/K$-system,
 \item $p_{\a,L}\sharp \Phi(x)$ is well-defined if $\Phi_\a^{-1}(x)L \neq \rho(\Phi_\a^{-1}(x)L)$. 
\end{enumerate}
\end{thm}

\begin{proof}
If $\a\in \R$, then $\Phi_\a$ is real and in $\rO(2n-1,1)$. So $\a, \ti L(x)$ satisfy \eqref{ba}. If $\a=is$ for some real $s$, then 
\[\Phi_{is}= \rho \Phi_{-is} \rho= \rho \overline{\Phi_{is}}\rho .\]
So if we write $\Phi_{is}= \bpm \eta_1& \eta_2\\ \eta_3& \eta_4\epm$, then $\eta_1, \eta_4$ are real and $\eta_2, \eta_3$ are pure imaginary matrices. A computation shows that $\a, \ti L(x)$ satisfy \eqref{ba}.  In other words, $p_{\a, \ti L(x)}$ satisfies the $U/K$-reality condition. 

To prove that $\ti\Phi(x)$ lies in $L^+(U)$, we expand $p_{\alpha,L}g_+p^{-1}_{\alpha,g_+^{-1}(\alpha)L}$ in a power series about $\lambda=\alpha$ and checking that it is in fact holomorphic and invertible there. The twisting condition ensures that the same is true at $\lambda=-\alpha$, hence $p_{\alpha,L}g_+p^{-1}_{\alpha,g_+^{-1}(\alpha)L}$ is a map into $L^+(U)$; unique factorisation finishes things off.

Let $\ti p= p_{\a, \ti L(x)}$.  Then 
\begin{equation}\label{eq:logdress}
\ti \Phi^{-1}\D\ti\Phi = \ti p\theta_\l \ti p^{-1} - \D\ti p \ti p^{-1}.
\end{equation}
Expand the above equality at $\l=\infty$, noting that $\theta_\l$ is a degree one polynomial in $\l$; we see that $\ti\Phi^{-1}\D\ti\Phi$ is degree one\footnote{$\lim_{\lambda\to\infty}\lambda^{-1}\tilde\Phi^{-1}\D\tilde\Phi=\begin{smatrix} 0&-\delta J\\\delta&0\end{smatrix}$ is bounded.} in $\l$ and is thus the Lax pair of a new solution of the $U/K$-system.
\end{proof}

We now write down explicit formulae for the action of $p_{\a, L}$ on $F$ and the frame of $F$. Recalling section \ref{sec:UK}, we note that a flat lift $F$ can be written
\[F=\Phi_1\stwovec{0}{g_2^{-1}\V c},\]
where $\Phi_\l$ is an extended frame and $\Phi_0=\bpm g_1&0\\ 0& g_2\epm$. Set 
\[\V m:=g_2^{-1}\V c = (m_1, \ldots, m_n), \qquad \Phi_1= (\vect{e}{n},\vect{u}{n}).\]
Then  
\[F= \sum_{j=1}^n m_j\V u_j,\]
and $\vect{e}{n}$ are principal curvature directions, $\V u_j$ is parallel to the curvature normal $v_j$ of $F$, and $\V u_j= m_j v_j$.    

Assume that $\alpha\neq\pm 1$. We cancel the factor of $p_{\alpha,L}(1)$ from the definition of the dressed frame $\ti \Phi(x)$ and write
\begin{equation}\label{bi}
\begin{aligned}
p_{\alpha, L}\sharp F:= \tilde F&=\Phi_1p^{-1}_{\alpha,\Phi^{-1}_\alpha L}(1)p_{\alpha,\Phi^{-1}_\alpha L}(0)\twovec{0}{g_2^{-1}\V c}\\
&=\Phi_1p^{-1}_{\alpha,\Phi^{-1}_\alpha L}(1)p_{\alpha,\Phi^{-1}_\alpha L}(0)\Phi_1^{-1}F.
\end{aligned}
\end{equation}
(Note that the factor $p_{\alpha,L}(1)\in \rO(2n-1,1)$ is an isometry of $\R^{2n-1,1}$, hence $\ti F$ is equal to $\ti\Phi_1\bpm 0\\ \ti g_2^{-1} \V c\epm$ up to an isometry.)

Introduce the notation $\Phi_\alpha^{-1}L=\ip{\twovec{W}{Z}}$, where choices are normalised such that $\nm W^2=-\nm Z^2=2$: it is easy to check that $\cl W=W$ and $\cl Z=\sgn(\alpha^2)Z$. We have
\[p^{-1}_{\alpha,\Phi_\alpha^{-1}L}(\lambda)=\begin{pmatrix}
  I+\frac{\alpha^2}{\lambda^2-\alpha^2}WW^T&-\frac{\alpha\lambda}{\lambda^2-\alpha^2}WZ^TJ\\
  \frac{\alpha\lambda}{\lambda^2-\alpha^2}ZW^T&I-\frac{\alpha^2}{\lambda^2-\alpha^2}ZZ^TJ
\end{pmatrix}, \qquad J= \I_{n-1,1},\]
from which we write
\[\ti\Phi_\l = \Phi_\l \bpm  I+\frac{\alpha^2}{\lambda^2-\alpha^2}WW^T&-\frac{\alpha\lambda}{\lambda^2-\alpha^2}WZ^TJ\\
  \frac{\alpha\lambda}{\lambda^2-\alpha^2}ZW^T&I-\frac{\alpha^2}{\lambda^2-\alpha^2}ZZ^TJ\epm.\]
Write
\[\ti\Phi_1= (\ti {\V e}_1, \ldots, \ti {\V e}_n, \ti {\V u}_1, \ldots, \ti{\V u}_n), \qquad \ti\Phi_0=\bpm \ti g_1& 0\\ 0& \ti g_2\epm,\]
and set
\[\ti F= \ti \Phi_1\bpm 0\\ \ti g_2^{-1} \V c\epm, \qquad \ti{\V m}:= \ti g_2^{-1} \V c = (\ti m_1, \ldots, \ti m_n).\]
Then 
\begin{gather}\label{bf}
\tilde{\V e}_i=\Phi_1 p^{-1}_{\alpha,\Phi_\alpha^{-1}L}(1)\Phi^{-1}_1\V e_i=
 \V e_i+\frac{\alpha W_i}{1-\alpha^2}\Phi_1\twovec{\alpha W}{Z},\\
\tilde{\V u}_i=\Phi_1 p^{-1}_{\alpha,\Phi_\alpha^{-1}L}(1)\Phi^{-1}_1\V u_i=\V u_i-\frac{\alpha\epsilon_i Z_i}{1-\alpha^2}\Phi_1\twovec{W}{\alpha Z},\\
\ti m_j = m_j + (Z, \V m)Z_j,\\ 
\ti F= F +\frac{(Z,\V m)}{1-\alpha^2}\left(\sum_{i=1}^n\alpha W_i\V e_i+Z_i\V u_i\right),
\end{gather}
where $W,Z$ are written with respect to the standard bases of $\R^n,\R^{n-1,1}$.

The above formulae imply that 
\begin{gather*}
\tilde F-F\,/\!\!/\,\tilde{\V e}_i-\V e_i\perp\tilde{\V u}_j-\V u_j\,/\!\!/\,\tilde{\V u}_i-\V u_i.
\end{gather*}
More is true, for $F$ and $\tilde F$ envelop of a congruence of $n$-spheres or $n$-hyperbolae and are, in fact, \emph{Ribaucour transforms} of each other (defined next). 

\begin{defn}
A \emph{congruence} of $n$-spheres ($n$-hyperbolae resp.) is a map into the space of $n$-spheres in a spaceform.
An enveloping submanifold of a congruence is a submanifold which has first-order contact with the congruence, i.e.\ each sphere (hyperbola resp.) is tangent to the submanifold it touches.
\end{defn}

\begin{rem}
An $n$-sphere in $\R^{2n-1,1}$ can be written as $\V c+ \{x\in V\n (x,x)= r^2\}$ for some space-like $(n+1)$-dimensional linear subspace $V$, $\V c\in V^\perp$, and a constant $r$. This $n$-sphere lies in the light-cone $\cL^{2n-1,1}$ if and only if $(\V c, \V c)= -r^2$. An $n$-hyperbola can be written as 
$\V c+ \{x\in V\n (x,x)= -r^2\}$ for some Lorentzian $(n+1)$-dimensional linear subspace $V$, $\V c\in V^\perp$, and a constant $r$. This $n$-hyperbola lies in $\cL^{2n-1,1}$ if and only if $(\V c,\V c)= r^2$. Note that the projections of both $n$-spheres and $n$-hyperbolas in $\cL^{2n-1,1}$ to $S^{2n-2}$ are $n$-spheres.\footnote{We will explain this more clearly in remarks \ref{rem:confinv}.}
\end{rem}

\begin{rem}
It can be easily seen that a generic $n$-dimensional congruence of $n$-spheres (or $n$-hyperbolae) has exactly two enveloping submanifolds $\cM, \cM^*$ and a map $\phi:\cM\to\cM^*$, so that for each $p\in\cM$ there is a $n$-sphere (or $n$-hyperbola) $C(p)$ in the congruence such that $\cM$ and $\cM^*$ are tangent to $C(p)$ at $p$ and $\phi(p)$ respectively. We will also call the map $\phi$ a congruence. The congruence $\phi$ is said to be \emph{Ribaucour} if the lines of curvature on $\cM$ map to lines of curvature on $\cM^*$. Otherwise said, the lines of curvature correspond and the tangent line through $p$ in the direction $\V e_i(p)$ meets the tangent line through $\phi(p)$ in the direction $\D\phi(\V e_i(p))$ at equal distance. This is the definition given of Ribaucour transform in \cite{Bruck2002} for isothermic surfaces.
\end{rem}

\begin{thm}\label{bh}
Let $\a\in \C$ and $L$ an isotropic line in $\C^{2n-1,1}$ satisfying \eqref{ba}, and $p_{\a,L}$ be the simple element in $L^-(U)$ defined by \eqref{bd}. Then the dressing action $F\mapsto \ti F= p_{\a, L}\sharp F$ defined by \eqref{bi} is a Ribaucour $n$-hyperbola congruence in $\cL^{2n-1,1}$ if $\a$ is real, and a Ribaucour $n$-sphere congruence in $\cL^{2n-1,1}$ if $\a$ is pure imaginary. Moreover, 
\begin{enumerate}
\item $\ti F-F \in \bpm \I_n&0\\ 0& \a^{-1}\I_n\epm \, \cL^{2n-1,1}$ and $\ti{\V e}_i-\V e_i$ is parallel to $\ti F-F$ for all $1\leq i\leq n$, where the $\V e_i$ and $\ti {\V e}_i$ are principal curvature directions for $F$ and $\ti F$ respectively,
\item If $F$ is a flat lift of a conformally flat immersion $f$ in $S^{2n-2}$, then $\ti F$ is a flat lift of another conformally flat immersion $\ti f$ in $S^{2n-2}$, and the transform $f\to \ti f$ is a Ribaucour transform in $S^{2n-2}$. 
\end{enumerate}
\end{thm}

\begin{proof}
Note that $F$ and $\ti F$ envelop an $n$-sphere (or a $n$-hyperbola) congruence if there exist vector fields $\xi$ normal to $F$ and $\ti\xi$ normal to $\ti F$ satisfying the following conditions:
\begin{enumerate}
\item $F+\xi = \ti F+\ti \xi$ and $(\xi, \xi)= (\ti \xi, \ti \xi)$,
\item $\li \V e_1(x), \ldots,\V e_n(x), \xi(x)\ri = \li \ti{\V e}_1(x), \ldots, \ti{\V e}_n(x), \ti \xi(x)\ri$, which will be denoted by $V(x)$,
\item $\ti F(x)-F(x)\in V(x)$.
\end{enumerate}
The above conditions imply that both $F$ and $\ti F$ are tangent to the  quadrics in the affine space $F(x)+ V(x)$,
$$(y-c(x), y-c(x)) = (\xi(x), \xi(x)),$$
at $F(x)$ and $\ti F(x)$,
where the centre $c(x)= F(x)+\xi(x)$.  If $V(x)$ is space-like then this quadric is an $n$-sphere, and if $V(x)$ is Lorentzian then this quadric is an $n$-hyperbola.  Since $F(x)$ is null, there exists a time-like $t_0(x)$ such that the quadric lies in the affine space $t_0(x)+V(x)$.  

We have seen that
\begin{align*}
& \ti F- F= \frac{(Z,\V m)}{1-\a^2}\left(\a\sum_j W_j\V e_j + \sum_j Z_j\V u_j  \right),\\
& \ti{\V e}_i-\V e_i= \frac{\a W_i}{1-\a^2}\left( \a\sum_j W_j\V e_j + \sum_j Z_j\V u_j\right),\\
& \ti{\V u}_i -\V u_i = -\frac{\a \e_i Z_i}{1-\a^2} \left( \sum_j W_j\V e_j + \a\sum_j Z_j\V u_j\right),
\end{align*} 
where $\ip{\!\bpm W\\ Z\epm\!}= \Phi_{\a}^{-1}L$ and is normalised so that $|W|^2=-|Z|^2=2$ and $\V m=g_2^{-1}\V c$. It follows from the formulae for $\tilde F-F$ and $\tilde{\V e}_i-\V e_i$ that
\[\tilde F-F=\frac{(Z,\V m)}{\alpha W_i}(\tilde{\V e}_i-\V e_i),\]
for each $i$. Equate the coefficients of the $\V e_i$ and $\V u_i$ in $\ti F- F= \xi-\ti \xi$ to get
\begin{equation}\label{be}
\sum_{i=1}^n\e_i \ti \xi_iZ_i= (\V m, Z), \qquad \xi_i =\ti \xi_i + (Z,\V m) Z_i,
\end{equation}
where $\xi= \sum_{i=1}^n \xi_i \V u_i$ and $\ti \xi= \sum_{i=1}^n \ti \xi_i \ti{\V u}_i$.  

Case 1: $\a\in\R$. Then $Z$ is real \eqref{ba}. Use the condition that ${\V e}_i(x)$ lies in $V(x)$ to see that $\xi$ has to be parallel to $\sum_j Z_j\V u_j$, hence $\xi= f\sum_j Z_j\V u_j$ for some real function $f$. It follows from $(\ti \xi, \ti \xi)=(\xi, \xi)$ and \eqref{be} that $f=\frac{1}{2}(Z,\V m)$.  So 
\[\xi_j=\frac{1}{2} (Z,\V m)Z_j= -\ti \xi_j.\]
Since $\xi$ is time-like, it follows that $V(x)$ is Lorentzian. This shows that $F$ and $\ti F$ envelop a congruence of $n$-hyperbolae and 
\[\xi= \frac{1}{2} (Z,\V m)\sum_{j=1}^n Z_j\V u_j, \qquad \ti \xi= -\frac{1}{2} (Z,\V m)\sum_{j=1}^n Z_j\ti{\V u}_j.\]

Case 2: $\a\in i\R$. Then $Z= i\g$ for some $\g\in \R^{n-1,1}$.  We can use the above argument to obtain 
\[\xi_j= -\ti \xi_j = -\frac{1}{2} (\g,\V m) \g_j.\]
Since $\g$ is space like, so is $V(x)$. In other words, $F$ and $\ti F$ envelop a congruence of $n$-spheres.  

It is easy to check that both the $n$-spheres and the $n$-hyperbolae lie in $\cL^{2n-1,1}$. As a consequence of Theorem \ref{thm:dress}, $\ti F$ is a flat lift of a new conformally flat $\ti f$ in $S^{2n-2}$ with identical line of curvature co-ordinates for $\ti F$.
\end{proof}

\def\calL{{\mathcal L}}
\def\o{\theta}

\begin{rems}\label{rem:confinv}
The discussion of congruences in $S^{2n-2}$ is particularly beautiful in the light-cone picture (e.g.\ \cite{Burstall2004}). A congruence of $n$-spheres may be viewed as a map $S:M^n\to G_{n-2}^+(\R^{2n-1,1})$ into the Grassmannian of positive definite $(n-2)$-planes; $\pr(S^\perp\cap\cL^{2n-1,1})\cong\pr(\cL^{n+1,1})\cong S^n$. The condition that $\ip F$ envelops $S$ then becomes very simple: $S\perp\ip{F,\D F}$. Burstall--Calderbank \cite{Burstall2004c} generalise the notion of Ribaucour in this setting by demanding simply that a general codimension congruence with two enveloping submanifolds $\ip F,\ip{\tilde F}$ is Ribaucour iff the bundle $\ip{F,\tilde F}$ is flat.

We may restate the above theorem in a more invariant manner, that views the hyperbola and sphere congruences as sub-quadrics of the quadric $\pr(\R^{2n-1,1})\cong S^{2n-2}$. Each $S^\perp(x):=V(x)\oplus F(x)=V(x)\oplus\tilde F(x)$ is a signature $(n+1,1)$-plane, hence $S:M^n\to G_{n-2}^+(\R^{2n-1,1})$ is an $n$-sphere congruence in the conformal $S^{2n-2}$, enveloped by $\ip F,\ip{\tilde F}:M^n\to\pr(\cL^{2n-1,1})\cong S^{2n-2}$. We may moreover calculate the flatness of the bundle $\ip{F,\hat F}$ to see that the enveloped congruence is indeed Ribaucour in the sense of Burstall--Calderbank \cite{Burstall2004c}.

Since the notion of enveloped sphere congruence is conformally invariant, the theorem is true in any Riemannian $S^{2n-2}\subset t_0^\perp$ we choose. Specifically: let $f\hookrightarrow S^{2n-2}$ be conformally flat with uniform multiplicity one, $F$ a flat lift and $\tilde f$ the projection of the transform $p_{\alpha,L}\#F$ by a simple element; then $f,\tilde f$ envelop a congruence of $n$-spheres and have corresponding curvature directions. To summarise, we have the following theorem:

\begin{thm}
Simple elements act by Ribaucour transforms on conformal flats with uniform multiplicity one.
\end{thm}
\end{rems}

Returning to flat lifts, we may also rephrase the construction of Ribaucour transforms in terms of a system of first order PDE: in particular, given a flat lift $F$, $\a\in \R^\times \cup i\R^\times$, and $\ell_0=\li Y_0\ri=\left\li \bpm W_0\\ Z_0\epm\right\ri$, such that $\a, \ell_0$ satisfy \eqref{ba}, the Ribaucour transform of $F$ by $p_{\a, \Phi_\a^{-1}(\ell_0)}$ may be constructed.

\begin{thm}
Let $F, \Psi, x, \Xi$ be as in Theorem \ref{thm:ndimconf} --- $\Psi_1=(\vect{e}{n},\vect{u}{n})$, $\Psi_0=\begin{smatrix}g_1&0\\0&g_2\end{smatrix}$, $F=\Phi_1\binom{0}{\V m}$, and $\V m=g_2^{-1}\V c$ --- and $\o_\l = \sum_{i=1}^n (\l a_i+ [a_i, \Xi])\D x_i$ the Lax pair of the solution $\Xi$ of the $U/K$-system. Given $\alpha\in\R^\times\cup i\R^\times$ and $\ell=\li Y_0\ri$ satisfying \eqref{ba}. Then:
\begin{enumerate}
\item The following system for $\C^{2n}$-valued maps $Y$ has a unique solution:
\begin{gather}\label{eq:hatl}
\D Y=-\theta_\alpha Y, \quad Y(0)= Y_0.    
\end{gather}
\item $\a$ and $\li Y(x)\ri$ satisfy \eqref{ba}, where $Y$ is the solution to \eqref{eq:hatl}.
\item Choose $W, Z$ so that $\li Y\ri = \left\li \bpm W\\ Z\epm \right\ri$ with $|W|^2=-|Z|^2=2$.  Then
\[F\mapsto \ti F: = F+ \frac{(Z,\V m)}{1-\a^2} \left(\sum_j aW_j\V e_j + Z_i\V u_i\right)\]
is the Ribaucour transform given in Theorem \ref{bh} by $p_{\a,\Phi_\a^{-1}(\ell_0)}$. 
\end{enumerate}
\end{thm}

\begin{proof}
System \eqref{eq:hatl} is solvable because $\theta_\alpha$ is flat. The rest follows.
\end{proof}

If we apply the dressing action of simple elements to the vacuum solution $\Xi=0$ repeatedly, then we can construct infinitely many families of explicit conformally flat $n$-immersions in $S^{2n-2}$ with uniform multiplicity one.

\subsection*{Permutability}

We may easily obtain a permutability theorem for Ribaucour transforms, or at least those that arise via simple element dressing. By theorem \ref{thm:dress}, combined with linear fractional transforms $x\mapsto\frac{x-\alpha,\beta}{x+\alpha,\beta}$ we see that
\[p_{\alpha,p_{\beta,M}(\alpha)L}p_{\beta,M}p_{\alpha,L}^{-1}\quad\text{and}\quad p_{\beta,p_{\alpha,L}(\beta)M}p_{\alpha,L}p_{\beta,M}^{-1},\]
are pole-free and invertible at $\pm\alpha,\pm\beta$ respectively. Putting these together and applying Liouville's theorem (holomorphic on $\pr^1\Rightarrow$ constant) we see that in fact
\begin{gather}\label{eq:perm}
p_{\alpha,p_{\beta,M}(\alpha)L}p_{\beta,M}=p_{\beta,p_{\alpha,L}(\beta)M}p_{\alpha,L}.
\end{gather}
Applied to our discussion, we see that given two Ribaucour transforms via simple elements, there exists a common fourth immersion which is simultaneously a Ribaucour transform of the first two (and is not the original immersion).

\section{Channel immersions}\label{sec:channel}

In this section, we consider conformally flat $n$-dimensional immersions into $S^{2n-2}$ with some multiplicity greater than one. The curvature distributions of such immersions have constant ranks and are smooth. Their flat lifts into $\cL^{2n-1,1}$ also have constant multiplicity. We show that in fact all but one curvature distribution has rank 1. Such submanifolds envelop a $p$-dimensional family of $(n-p)$ dimensional spheres; they are the analogues of the channel hypersurfaces in $S^4$, and hence will be called \emph{channel immersions}. Unlike the uniform multiplicity one case, we do not know whether line of curvature co-ordinates exist for such immersions. If line of curvature co-ordinates do exist, then the Gauss-Codazzi equations for such an immersion is the $U/K$-system defined by a non-semisimple maximal abelian algebra in $\fp$. Conversely, solutions to these $U/K$-systems give rise to conformally flat immersions with one multiplicity $\ge 2$. 

Recall first theorem \ref{thm:Ffcurvs}, which says that a conformal flat $f$ and any flat lift $F$ have identical curvature distributions.

\begin{thm}\label{thm:channel1}
Suppose that $f$ is conformally flat with flat normal bundle and constant multiplicities, with at least one multiplicity $k\ge 2$. Then the curvature distributions are smooth and there is precisely \emph{one} curvature distribution with $\rank\ge 2$ so that $f$ has multiplicity $(1,\ldots,1,k)$. We may therefore write
\[TM=\bigoplus_{i=1}^pE_i\oplus E
,\]
where $\rank E_i=1$, $\rank E=k=n-p$.

Let $F$ be any flat lift of $f$ and let $v_1,\ldots, v_p, v$ be the curvature normals of $F$. The $v_i$ are space-like and orthogonal, $v$ is isotropic and orthogonal to the $v_i$, and all are non-zero. Any flat lift $F$ has degenerate normal bundle and the formulae of theorem \ref{thm:Ffcurvs} relating curvature normals of $f$ and $F$ still hold. Moreover the distribution $E$ is integrable, and the leaf of $E$ through any point is contained in a copy of $S^{n-p}\subset S^{2n-2}$. Indeed the repeated curvature normal $v^\R$ of $f$ is a parallel section of $N_f\oplus\ip f$ over $E$ and the $(n-p)$-sphere in question has (Euclidean) radius $\frac{1}{\nm{v^\R}}<1$. 
\end{thm}

\begin{proof}
First recall, from theorem \ref{thm:Ffcurvs}, that $f$ and any flat lift $F$ share the same curvature distributions, and that any distribution of rank $\ge 2$ has an isotropic, non-zero curvature normal for $F$. If there are two such then they must be scalar multiples, since two orthogonal non-zero isotropic vectors contradict the fact that maximal isotropic subspaces of $\R^{2n-1,1}$ are lines. The part of flat differentiation $\D$ that maps $\ip{\D F}\leftrightarrow N_F$ between tangent and normal bundle is well-known to be a $\fo(2n-1,1)$-valued 1-form $\cN$ such that $\II_F=\cN\D F$. Since $F$ is parallel in $N_F$, it follows that $\D F=\cN F$. Applying this to the supposition that there are two isotropic curvature normals which are non-trivial multiples of each other gives a contradiction.


For the remainder, we appeal to a theorem of Terng \cite{Terng1987} which states that the curvature distributions of $f$ are integrable and that the leaf of $E$ through any point is an open subset of a $(n-p)$-plane or an $(n-p)$-sphere. Since, for us, the leaf must lie in $S^{2n-2}$, we necessarily have (part of an) $(n-p)$-sphere. Indeed one may see that $f+\frac{v^\R}{\nm{v^\R}^2}$ is constant on any leaf of $E$, hence the $(n-p)$-sphere has radius $\frac{1}{\nm{v^\R}}$: since $v^\R$ is parallel, this radius is independent of $E$.
\end{proof}

The following theorems can be proved in the same way as for Theorems \ref{thm:ndimconf} and \ref{thm:ndimconf2}.

\begin{thm}\label{thm:channel}
Let $f,\lst{E}{p},E$ be as in Theorem \ref{thm:channel1}, $F$ be a flat lift of $f$, and $\lst{v}{p}$, $v$ the corresponding curvature normals of $F$. Suppose that $F$ is parameterised by line of curvature co-ordinates $(x_1,\ldots, x_n)$. Then:
\begin{enumerate}
\item There exists an $\rO(2n-1,1)$ frame $\Phi=(\vect{e}{n},\vect{u}{n})$ with $\V e_1,\ldots,\V e_n$ principal curvature directions, and $\V u_i= v_i/||v_i||$ for $1\leq i\leq p$, and $v=\V u_{n-1}+\V u_{n+1}$,
\item  $\Phi^{-1}\D \Phi= \bpm A & \delta\\ -J \delta^T &B\epm$, where 
$\delta=\sum_{i=1}^p e_{ii} \D x_i + \sum_{j=p+1}^n (e_{j, n-1} - e_{jn}) \D x_j$.
\item Set 
$ a_i = \bpm 0 & e_{ii}\\ -J e_{ii} &0\epm$ for $i\leq p$, and 
$a_j= \bpm 0& e_{j,n-1}- e_{jn} \\-(e_{n-1,j} +e_{nj})&0\epm$ for $p+1\leq j\leq n$.
Then $\fa_p= \ip{a_i}_{i=1}^n$ is a non-semisimple maximal abelian subalgebra in $\fp$ and $D= \sum_{i=1}^n a_i \D x_i$.
\item There exists a map $\Xi:M\to \fa_p^\perp\cap \fp$ such that 
$\bpm A&0\\ 0& B\epm = \sum_{i=1}^n [a_i, \Xi]\D x_i$.
In other words, $\Xi$ is a solution of the $U/K$-system defined by $\fa_p$.
\item There exists a constant null vector $\V c\in \R^{n-1,1}$ such that $F= \Phi \bpm 0\\ g_2^{-1}\V c\epm$. Set $\V y:= g_2^{-1}\V c= (y_1, \ldots, y_n)^T$. Then $\I_F= \sum_{j=1}^n y_j^2 \D x_j^2$.  
\end{enumerate}
\end{thm}

\begin{thm}
Given a solution $\Xi$ of the $U/K$-system defined by $\fa_p$ and a constant null vector $\V c\in \R^{2n-1,1}$, let $\Phi_\l$ be an extended flat frame for the Lax pair of $\Xi$, then:
\begin{enumerate}
\item[(i)] $\Phi_0=\bpm g_1&0\\ 0& g_2\epm$.
\item[(ii)] Write $\V y:=g_2^{-1}\V c= (y_1, \ldots, y_n)^T$ and $\Phi_1=(\vect{e}{n},\vect{u}{n})$. Then $F= \Phi_1\bpm 0\\ \V y\epm$ is a flat immersion with degenerate flat normal bundle  and constant multiplicities, $F$ is parameterised by line of curvature co-ordinates,
\[\I_F= \sum_{i=1}^n y_i \D x_i^2,\qquad \II_F= \sum_{i=1}^p y_i \D x_i^2 \V u_i +\sum_{j=p+1}^n y_j \D x_j^2\, v,\]
and the curvature normals are $v_i= y_i^{-1}\V u_i$ for $1\leq i\leq p$ and $v=\V u_{n-1}+\V u_n$. 
\end{enumerate} 
\end{thm}

The discussion of dressing and Ribaucour transforms goes through exactly as in section \ref{sec:rib} for channel immersions that have line of curvature co-ordinates. Since, by \eqref{eq:logdress}, logarithmic derivatives of dressed frames have the same $\fp$-part, and thus similar second fundamental forms, it is clear that dressing a channel hypersurface yields another. Similarly, by the correspondence of theorem \ref{thm:channel}, we also get dressing and Ribaucour transforms of solutions to the $U/K$-system defined by non-Cartan maximal subalgebra $\fa_p$.

We may also repeatedly apply the dressing action of $p_{\a, L}$ to the vacuum solution $\Xi=0$ to construct infinitely many families of conformally flat channel immersions.  These immersions are given by explicit formulae because the extended frame for the vacuum solution is $\exp(\sum_{i=1}^n a_i \l x_i)$.

\newcommand{\noopsort}[1]{}


\end{document}